\documentclass[12pt,oneside,reqno]{amsart}
\usepackage{amsfonts,amssymb,amscd,amsmath,latexsym,enumerate,verbatim,calc, quotmark, mathtools}
\usepackage[all]{xy}

\usepackage[pagebackref]{hyperref}
\hypersetup{
     colorlinks = true, linkcolor = red,
     citecolor   = blue,
     urlcolor    = blue,
}

\DeclareMathOperator*{\Spec}{Spec}
\def\NZQ{\mathbb}               
\def\NN{{\NZQ N}}

\def\Um{{\operatorname{Um}}}          
\def\E{{\operatorname{E}}}            
\def\GL{{\operatorname{GL}}}          
\def\M{{\operatorname{M}}}            
\def\Um{{\operatorname{Um}}}          
\def\vpmod{\!\!\pmod} 
\DeclareMathOperator*{\height}{height} 
            
\DeclareMathOperator*{\Trans}{Trans}
\DeclareMathOperator*{\ETrans}{ETrans}
\DeclareMathOperator*{\End}{End}
\DeclareMathOperator*{\Hom}{Hom}
\DeclareMathOperator*{\Aut}{Aut}

\newtheorem{theorem}{Theorem}[section]
\newtheorem{lemma}[theorem]{Lemma}
\newtheorem{corollary}[theorem]{Corollary}
\newtheorem{proposition}[theorem]{Proposition}
\newtheorem{remark}[theorem]{Remark}

\newtheorem{definition}[theorem]{Definition}

\newtheorem{question}[theorem]{Question}

\newtheorem*{acknowledgement}{Acknowledgement}

\begin{document}

\title{\small on the existence of unimodular elements and cancellation of projective modules over noetherian and non-noetherian rings}

\author{{\bf Anjan Gupta}\\
Tata Institute of Fundamental Research, Mumbai }

\footnote{Correspondence author: Anjan Gupta; 
{\it email: anjan@math.tifr.res.in, agmath@gmail.com}}

\maketitle

\noindent {\it Mathematics Subject Classification: 13B25, 13C10.}

\vskip3mm

\noindent {\it Key words : unimodular elements, transvections, projective modules.}
\subjclass{}

\vskip3mm

\begin{abstract}  Let $R$ be a commutative ring of dimension $d$, $S = R[X]$ or $R[X, 1/X]$ and $P$ a finitely generated projective $S$ module of rank $r$. Then $P$ is cancellative if $P$  has a unimodular element and $r \geq d + 1$. Moreover if  $r \geq \dim (S)$ then $P$ has a unimodular element and therefore  $P$ is cancellative. As an application we have proved that if $R$ is a ring of dimension $d$ of finite type over a Pr\"{u}fer domain and $P$ is a projective $R[X]$ or $R[X, 1/X]$ module of rank at least $d + 1$, then $P$ has a unimodular element and is cancellative.  
\end{abstract}

\section{introduction}
A projective $R$-module $P$ is said to have a unimodular element if $P = R \oplus Q$ for some submodule $Q$ of $P$. $P$ is called cancellative if $Q \oplus P \cong Q \oplus P'$ for some projective $R$-modules $P', Q$ implies that $P \cong P'$. 

Now assume that $R$ is a commutative noetherian ring of dimension $d$ and $P$ a finitely generated projective $R$ module of rank  $r \geq d + 1$. A classical result of Serre  asserts that $P$  has a unimodular element. Subsequently Bass in \cite{bass} proved that  $P$ is also cancellative. Much later Bhatwadekar-Lindel-Rao in \cite{blr} proved that if $P$ is a finitely generated projective module over $R[X_1, \ldots, X_n, Y_1^{\pm 1}, \ldots, Y_m^{\pm1}]$ of rank $r \geq d + 1$, $d = \dim(R)$, then $P$ has a unimodular element.  Using their result Rao in \cite{rao} showed that projective modules over $R[X_1, \ldots, X_n]$ of rank $r \geq \{d + 1, 2\}$, $d = \dim(R)$ are  cancellative. Finally in \cite{lind}, Lindel extended Rao's result by showing that projective modules over $R[X_1, \ldots, X_n, Y_1^{\pm 1}, \ldots, Y_m^{\pm1}]$ of rank $r \geq \{d + 1, 2\}$, $d = \dim(R)$ are  also cancellative.

Just after the solution of Serre's conjecture by Quillen and Suslin, work started to explore projective modules over a non-noetherian base ring $R$. An early result of Brewer and Costa in \cite{brecosta} asserts that if $R$ is a ring of dimension zero, then any finitely generated projective module over $R[X_1, X_2, \ldots, X_n]$ is extended from $R$. Surprisingly, Heitmann in \cite{heit} proved that  both Serre's splitting theorem and Bass's cancellation theorem are true for rings not necessarily noetherian. This leads us to investigate whether similar phenomena occur for projective modules over $R[X_1, \ldots, X_n, Y_1^{\pm 1}, \ldots, Y_m^{\pm1}]$ over non-noetherian base rings $R$.
   
However not much was known in this direction until a recent result of Yengui in \cite{yengui} which states that any stably free module over $R[X]$ of rank $r \geq d + 1$ is cancellative where $d$ is the dimension of the base ring $R$. Later  Abedelfatah in \cite{abed} proved that a stably free module over $R[X, 1/X]$ is also cancellative if its rank $r \geq d + 1$, $d$ being the dimension of  $R$. In this article we show the following.

\begin{theorem}{\rm(Theorems \ref{u15}, \ref{u13})}\label{1}
Let $R$ be a ring of dimension $d$ and $S = R[X]$ or $R[X, 1/X]$. Let $P$ be a finitely generated projective $S$-modules of rank $r \geq \dim(S)$. Then the following are true.
\begin{itemize}
\item[1.]
The natural map $\Um(P) \rightarrow \Um(P/XP)$ is surjective when $S = R[X]$.
\item[2.]
The natural map $\Um(P) \rightarrow \Um(P/(X - 1)P)$ is surjective when $S = R[X, 1/X]$.
\end{itemize}
\end{theorem}

Note that the above theorem together with Heitmann's result show that $P$  has a unimodular element. If $R$ is noetherian, then $\dim(S) = \dim(R) + 1$. Therefore, our result generalizes the corresponding result of Lindel for noetherian rings in \cite{lind}. A result of Seidenberg in  \cite{seidenberg} states that $\dim(S)$ can be any number in between $d + 1$ and $2d + 1$. So our lower bound for the rank of $P$ in Theorem \ref{1} is much weaker than expected. However if $P$ has a unimodular element, then $P$ is cancellative if its rank $r \geq \dim(R) + 1$. To show this we  prove the following.

\begin{theorem}\label{2}{\rm(Theorem \ref{ce18})}
Let $R$ be a ring of dimension $d$, $I$ an ideal of $R$ and $S = R[X]$ or $R[X, 1/X]$. Let $P$ be a finitely generated projective $S$-module of rank $r \geq d$ and $Q = S^2 \oplus P$. Then $\ETrans(Q, I)$ acts transitively on $\Um(Q, I)$.
\end{theorem}

Theorems \ref{1} and \ref{2} together show the following.

\begin{theorem}{\rm(Theorem \ref{mr2})}\label{3}
Let $R$ be a ring of dimension $d$ and $S = R[X]$ or $R[X, 1/X]$. Let $P$ be a finitely generated projective $S$-module of rank $r$.
Then $P$ is cancellative in the following cases. 

\begin{itemize}
\item [(1)] $P$ has a unimodular element and $r \geq d + 1$.

\item [(2)] Rank of $P$ is at least equal to $\dim(S)$.

\end{itemize}
\end{theorem}

It remains to see whether $P$ has a unimodular element if $r \geq d + 1$, $d$ being the dimension of $R$. As an application we have shown the following.

\begin{theorem}{\rm(Theorem \ref{mr3.5})}\label{3.5}
Let $R$ be a ring  of dimension $d$  of finite type over a Pr\"{u}fer domain and  $S = R[X]$ or $R[X, 1/X]$. Let $P$ be a finitely generated projective $S$-module of rank $r \geq d + 1$. Then the following holds.

\begin{itemize}
\item[1.] If $S = R[X]$, then the natural map $\Um(P) \rightarrow \Um(P/XP)$ is surjective.

\item[2.] If $S = R[X, 1/X]$, then the natural map $\Um(P) \rightarrow \Um(P/(X - 1)P)$ is surjective.
\end{itemize}
In particular $P$ has a unimodular element.  Moreover if $P'$ is another projective $S$-module of rank $r$ and $Q \oplus P \cong Q \oplus P'$ for some projective $S$-module $Q$, then $P \cong P'$.
\end{theorem}

In section 2, we revisit the result of Heitmann in \cite{heit} giving a simpler algebraic proof (see Theorem \ref{P6}). We show that if $\Spec(R) = V(s) \sqcup D(s)$, each of $V(s)$ and $D(s)$ has dimension at most $d$, then any finitely generated projective $R$-module $P$ of rank $r \geq d + 1$ has a unimodular element and is also cancellative (Corollary \ref{P9}). In section 3 we introduce the important notion of transvections. We prove Theorem \ref{2} whose method leans heavily on the work of Roitman \cite{roit}. One of the key techniques to prove the corresponding result for noetherian rings in \cite{lind} is to find a nonzero divisor $s$ such that $P_s$ is free, which is not available in our case. We avoid this difficulty by introducing the notion of a Lindel pair motivated by the work of Lindel in \cite{lind}. In section 4 we prove Theorem \ref{1}. In final section 5,  the main result of this article is proved. We recall the notion of strong S-ring. A result of Malik and Mott in \cite{malik} enables us to apply our main result to prove Theorem \ref{3}. Our paper contains numerous questions which we hope, will stimulate interest in the near future. The reader who would like to follow the story further is encouraged to browse Lam's excellent book (\cite{lam}, Chapter VIII, \S 7).
   
{\bf All rings in this article are assumed to be commutative with unity $1 \not = 0$ not necessarily noetherian and all projective modules are finitely generated.}
\section {Preliminaries}
In this section we shall give a few definitions and prove some elementary lemmas to prove the main results in the later sections. As said earlier all rings are commutative and possibly even non-noetherian.
%
The following is proved in (\cite{clmr}, Theorem 8). We give here a short algebraic proof.

\begin{lemma}\label{P3}
Let $R$ be a ring and $I(a) = (a) + (\sqrt{0} : a)$. Then the following are equivalent.
\itemize
\item[(1)] 
$R$ has dimension at most $d$.
\item[(2)]

The quotient ring $R/I(a)$ has dimension at most $d-1$ for all $a \in R$.
\end{lemma}

\begin{proof}
We first prove that (1) implies (2). It is enough to show that the height of $I(a)$ is at least one for for all $a \in R$. If possible assume that $I(a) \subset \mathfrak{p}$ for some minimal prime ideal $\mathfrak{p}$ of $R$. Now $(R_{\mathfrak{p}} , \mathfrak{p}R_{\mathfrak{p}})$ is a zero dimensional local ring 
and $I(a)R_{\mathfrak{p}} = (a) + (\mathfrak{p}R_{\mathfrak{p}} : a) \subset \mathfrak{p}R_{\mathfrak{p}}$. This is absurd as $a \in \mathfrak{p}R_{\mathfrak{p}}$ gives $(\mathfrak{p}R_{\mathfrak{p}} : a) = R_{\mathfrak{p}}$.

Now we assume that (2) holds. If possible let $\dim(R) \geq d + 1$ and $\mathfrak{p}_0 \subset \mathfrak{p}_1 \subset  \ldots \subset \mathfrak{p}_{d + 1}$ be an ascending chain of prime ideals of length $d + 1$. Choose $a \in \mathfrak{p}_1  - \mathfrak{p}_0$. Then $(\sqrt{0} : a) \subset \mathfrak{p_0}$. Therefore, $I(a) \subset \mathfrak{p}_1$ and $\dim (R / I(a)) \geq d$ which is a contradiction.
\end{proof}

\begin{definition}{\rm(Unimodular element, Order ideal)}\label{P1}
Let $P$ be a  projective module over a ring $R$. We shall call an element $p \in P$ unimodular if there exists  $\phi \in \Hom(P, R)$ such that $\phi(p) = 1$. In other words we have a submodule $Q$ of $P$ such that $P = Rp \oplus Q$. The set of unimodular elements in $P$ is denoted by $\Um(P)$. If $I$ is an ideal of $R$ and $P = R \oplus Q$, then  $\Um(P, I)$ denotes the set of all unimodular elements in $P$ which are $(1, 0)$ modulo $I$. For $p \in P$ we define $O(p) = \{\phi(p) : \phi \in \Hom(P, R)\}$ to be the order ideal of $p$. Clearly $p \in \Um(P)$ if and only if $O(p) = R$.
\end{definition}

\begin{definition}{\rm(Unimodular row, Elementary action)}\label{P4}
A row $(a_1, a_2, \ldots, a_n) \in R^{n}$ is said to be unimodular if there exists another row $(b_1, b_2, \ldots, b_n) \in R^{n}$ such that $\sum_{i =1}^{n}a_ib_i = 1$ i.e. $(a_1, a_2, \ldots, a_n) \in \Um(R^n)$. The set of unimodular rows of length $n$ is denoted by $\Um_{n}(R)$. We define $\Um_n(R, I)$ as the set $\{v \in \Um_n(R) : v \equiv e_1 \vpmod I\}$ for any ideal $I$ of $R$.

Any subgroup $G$ of $\GL_n(R, I) = \{ \alpha \in \GL_n(R) : \alpha \equiv I_n \vpmod I \}$ acts on $\Um_n(R, I)$ where $I_n$ denotes the identity matrix.  Let $v, w \in \Um_n(R, I)$, we write $v \sim_G w$ if $v = wg$ for some $g \in G$. Given $\lambda \in R$, for $i \neq j$, let $E_{ij}(\lambda) = I_n + \lambda{e_{ij}}$,  $e_{ij} \in \M_n(R)$ is the matrix  whose only nonzero entry is $1$ at the $(i,j)$-th position.  Such $E_{ij}(\lambda)$'s are called elementary matrices.  The subgroup of $\GL_n(R)$ generated by  $E_{ij}(\lambda), i \not= j, \lambda \in R$ is called the elementary subgroup of $\GL_n(R)$ and denoted  by $\E_n(R)$. Similarly $\E_n(I)$ is defined as the subgroup generated by $E_{ij}(\lambda), i \not= j, \lambda \in I$ for any ideal $I$ of $R$. We define $\E_n(R, I)$  to be the normal closure of $\E_n(I)$ in $\E_n(R)$. It is the smallest  normal subgroup of $\E_n(R)$ containing the element $E_{21}(x), x \in I$
\end{definition}

If $I$ is an ideal of $R$, then we define $V_R(I) = \{ \mathfrak{p} \in \Spec(R) ; I \subset \mathfrak{p}\}$. The following is a stronger version of (\cite{abed}, Proposition 2.3).

\begin{theorem}\label{P5}
Let $S$ be a multiplicative closed set in a ring $A$ and $R = S^{-1}A$  a ring of dimension at most $d$. Let $(a_0, a_1, \ldots, a_n) \in \Um_{n +1}(R), n \geq d + 1$. Then for any $s \in S$ there exists $c_1, c_2, \ldots, c_{d +1} \in sA$ such that  $(a_1 +  c_1a_0, a_2 + c_2a_0, \ldots, a_{d + 1} + c_{d + 1}a_0, a_{d + 2}, \ldots, a_n) \in \Um_{n}(R)$.   
\end{theorem}

\begin{proof}
Multiplying $(a_0, a_1, \ldots, a_n)$ by suitable $s \in S$ , it is enough to assume that $(a_0, a_1, \ldots, a_n) \in A^n$. Also replacing $A$ by $A/ (a_{d +2}, a_{d +3}, \ldots, a_n)$ we may assume that $n = d + 1$. We prove the result by induction on $d$. 

Assume $d = 0, n =1$ i.e. $(a_0, a_1) \in \Um_2(R)$. Then by Lemma \ref{P3} we have $I(\frac{a_1}{1}) = I(a_1)R = R$. So $a_1b + c \in S$ for some $c \in (\sqrt{0} : a_1)$. We claim that $V_R(a_1 + sca_0) =V_R(a_1, s) \cup V_R(a_1, c) \cup V_R(a_1, a_0)$. If $a_1 + sca_0 \in \mathfrak{p}$ for some prime ideal $\mathfrak{p}$, then we also have $a_1 sca_0  \in \mathfrak{p}$. This means that both $a_1, sca_0 \in \mathfrak{p}$. Therefore, $\mathfrak{p}$ contains one of the ideals $(a_1, s), (a_1, c)$ and $(a_1, a_0)$. So we have $V_R(a_1 + sca_0) \subset V_R(a_1, s) \cup V_R(a_1, c) \cup V_R(a_1, a_0)$. The reverse inclusion is clear. We have $(a_1, s) = (a_1, c) = (a_1, a_0) = R$. Therefore, $V_R(a_1 + sca_0) = \emptyset$ i.e. $(a_1 + c_1a_0) \in \Um_2(R)$ for $c_1 = sc$.

 Let overline denote the reduction modulo $I(a_{d +1}) = (a_{d+1}) + (\sqrt{0} : a_{d+1})$. Then we have $\overline{R} = \overline{S}^{-1}\overline{A}$ and $\dim ( \overline{R}) \leq d - 1$ by Lemma \ref{P3}. So by the induction hypothesis, we have $c_1, c_2, \ldots, c_d \in sA$ such that $(\overline{a_1} +  \overline{c_1a_0}, \overline{a_2} + \overline{c_2a_0}, \ldots, \overline{a_{d }} + \overline{c_{d}a_0}) \in \Um_d(\overline{R})$. Let  $J = (a_1 +  c_1a_0, a_2 + c_2a_0, \ldots, a_{d} + c_{d}a_0)$. Then there exists $s' \in S$ such that $s' \in J + I(a_{d +1})$. So $s' \equiv a_{d+1}b + c\  \vpmod J$ for some $c \in (\sqrt{0} : a_{d +1})$. By argument similar to the case when $d = 0$, we have $V_R(J + (a_{d+1} + sca_0)) = V_R(J + (a_{d+1}, s)) \cup V_R(J + (a_{d+1},  c)) \cup V_R(J + (a_{d+1}, a_0))$. Clearly $J + (a_{d+1},  s) = R$. Also $J + (a_{d+1},  c) = R$ by the choice of $c$ and $J + (a_{d+1},  a_0) = (a_0, a_1, \ldots, a_{d + 1}) = R$.

Therefore, $V_R(J + (a_{d+1} + sca_0)) = \emptyset$ and $(a_1 +  c_1a_0, a_2 + c_2a_0, \ldots, a_{d + 1} + c_{d + 1}a_0) \in \Um_{d +1}(R)$ for $c_{d +1} = sc$.
\end{proof}

A weaker version of the  following when $R = S$ in the theorem below  was proved by Heitmann in (\cite{heit})
using the patch topology on the Spectrum of a ring. A careful study of his result gives the following short algebraic proof.

\begin{theorem}\label{P6}
Let $S$ be a multiplicative closed set in a ring $A$ and $R = S^{-1}A$  a ring of dimension at most $d$. Let $Q$ be a finitely presented $A$-module such that $P = S^{-1}Q$ is a projective $R$-module of rank $r \geq d + 1$. Suppose $(a, x) \in \Um(R \oplus P)$ for some $(a, x) \in A \oplus Q$ and $Q = Ax + Q' $ for some finitely generated $A$ submodule $Q'$ of $Q$. Then for any $s \in S$ we have $ y \in Q'$ such that $x + say \in \Um(P)$.
\end{theorem}

\begin{proof}
We see that $(a, x) \in  \Um(R \oplus P)$ gives $(sa, x) \in  \Um(R \oplus P)$. Therefore,  given  $(a, x) \in \Um(R \oplus P)$ we only need to find  $y \in Q'$ such that $x + ay \in \Um(P)$.

Let $Q'$ be generated by  $x_1, x_2, \ldots, x_n$. We recall $O(x) = \{\phi(x) : \phi \in \Hom(Q, A)\}$. Since $Q$ is finitely presented, $O(x)$ commutes with localization i.e. $O(x/1) = O(x)R$. If $O(x) \cap S \not= \emptyset$ then $x \in \Um(P)$. So we choose $y = 0$ and we are done. 

We now assume that $O(x) \cap S = \emptyset$. Let $\mathfrak{p}$ be any prime ideal of $A$ containing $O(x)$ and $\mathfrak{p} \cap S = \emptyset$. Note that $P$ is generated by $x, x_1, \ldots, x_n$ as an $R$-module and the image of $x$ in  $P \otimes_R R_{\mathfrak{p}R}/ \mathfrak{p}R_{\mathfrak{p}R}$ is zero. So $P \otimes_R R_{\mathfrak{p}R}/ \mathfrak{p}R_{\mathfrak{p}R}$ is a vector space generated by $x_1, x_2, \ldots, x_n$. We choose $f \in A - \mathfrak{p}$ such that $P_{f}$ is a free $R_{f}$-module generated by a subset of $\{x_1, x_2, \ldots, x_n\}$ as basis. So we can cover the Zariski open set $V_R(O(x)R)$ by a finite number of basic open sets $D(f_i/1), f_i \in A$ such that on each $D(f_i/1)$ a subset of $\{x_1, x_2, \ldots, x_n\}$ will generate $P_{f_i}$ as a basis. Since $V_R(O(x)R) = \cup_{i =1}^n D(f_i/1)$, we have $$O(x)R + (f_1, f_2, \ldots, f_n)R = R.$$
Now $P_{f_1}$ is free on a subset of $\{x_1, x_2, \ldots, x_n\}$ and  $(a, x) \in \Um(R_{f_1} \oplus P_{f_1})$. So by Theorem \ref{P5} we have $u_1 \in Q'$ such that $x + af_1u_1\in \Um(P_{f_1})$. Therefore, $f_1 \in \sqrt{O(x + af_1u_1)R}$ which gives $O(x)R \subset \sqrt{O(x + af_1u_1)R}$. Thus we have
$$O(x + af_1u_1)R + (f_2, \ldots, \\f_n)R = R.$$
Repeating the above argument we have $u_2, u_3, \ldots, u_n \in Q'$ such that $O(x + af_1u_1 + af_2u_2 + \ldots + af_nu_n)R = R$. If $y = f_1u_1 + f_2u_2 + \ldots f_n u_n$ then $y \in Q'$ and $x + ay \in \Um(P)$.
\end{proof}

\begin{corollary}\label{P8}
Let $\Spec(R) = V(s) \sqcup D(s), s \in R$ be such that both $V(s)$ and $D(s)$ have dimension at most $d$. Let $P = Rp + P'$ be a projective $R$-module of rank $r \geq d +1$ such that $(a, p) \in \Um(R \oplus P)$. Then there exists $q \in P'$ such that $p + aq \in \Um(P)$.
\end{corollary}

\begin{proof} 
By Theorem \ref{P6} we have $q_1 \in P'$ such that $p + aq_1$ is a unimodular element modulo $s$. Now $(a,  p + aq_1) \in \Um(P_s)$. So by another application of Theorem \ref{P6} we have $q_2 \in P'$ such that $p + aq_1 + saq_2 \in \Um(P_s)$. Clearly $p + a(q_1 + sq_2) \in \Um(P)$ as it is unimodular modulo $s$ as well as in the localization at $s$.
\end{proof}

The following is a stronger version of (\cite{heit}, Corollaries 2.6, 2.7) stated in terms of Krull dimension.

\begin{corollary}\label{P9}
 Let $\Spec(R) = V(s) \sqcup D(s), s \in R$ such that both $V(s)$ and $D(s)$ have dimension at most $d$ and $P = Rp + P_1$ a projective module of rank $r \geq d + 1$. Then $P$ has a unimodular element of the form $p + q$ for some $q \in P_1$. Moreover if $P \oplus Q \cong P' \oplus Q$ for some projective modules $P', Q$, then $P \cong P'$.
\end{corollary}

\begin{corollary}\label{P9e1}
Let $R$ be a zero dimensional ring. Then any projective module over $R$ is free.
\end{corollary}

For any $f \in R[X, 1/X]$, we denote by $hc(f)$ and $lc(f)$ the coefficients of the highest degree and the lowest degree terms in $X$ respectively. A polynomial $f \in R[X]$ is said to be monic if $hc(f) = 1$. If $g \in R[X, 1/X]$ then $g$ is said to be bimonic if $hc(g) = lc(g) = 1$. The following is easy. See (\cite{lam}, Chapter III, Lemma 1.1) and (\cite{abed}, Proposition 3.3).

\begin{lemma}\label{u1}
Let $S = R[X]$ or $R[X, 1/X]$ and $I, J$  ideals of $S, R$ respectively such that $I + J S = S$. Then the following hold.
\begin{itemize}
\item[1.]
If $S = R[X]$ and $I$ contains a monic polynomial then $I \cap R + J = R$.
\item[2.]
If $S = R[X, 1/X]$ and $I$ contains a bimonic polynomial then $I \cap R + J = R$.
\end{itemize}
\end{lemma}

The following follows from (\cite{kap}, Chapter 1, Section 5, Exercise 3).

\begin{lemma}\label{u2}
Let $S = R[X]$ or $R[X, 1/X]$. Let $P \in \Spec(R)$ be of finite height and $Q \in \Spec(S)$  such that $P = Q \cap R$. Then $$\height(Q) - \height(PS) \leq 1.$$
\end{lemma}

%
%

\begin{lemma}\label{u3}
Let $R$ be a ring of finite dimension and $S = R[X]$ or $R[X, 1/X]$. Let $\mathfrak{m}$ be a maximal ideal of $S$ such that $\height (\mathfrak{m}) = \dim(S)$. Then $\mathfrak{n} = \mathfrak{m} \cap R$ is a maximal ideal of $R$.
\end{lemma}

\begin{proof}
If possible assume that $\mathfrak{n}$ is not a maximal ideal of $R$. Then $\mathfrak{n}\subsetneq \mathfrak{n}_1$ for some prime ideal $\mathfrak{n}_1$ of $R$. We have a prime ideal $\mathfrak{m}_1$ of $S$ lying above $\mathfrak{n}_1$ i.e. $\mathfrak{m}_1 \cap R = \mathfrak{n}_1$ and $\mathfrak{n}_1S \subsetneq   \mathfrak{m}_1$. Then by Lemma \ref{u2} we have $\height (\mathfrak{m}_1) = \height (\mathfrak{n}_1S) + 1 \geq   \height (\mathfrak{n}S) + 2 \geq \height(\mathfrak{m}) + 1 = \dim (S) + 1 $ which is a contradiction.
\end{proof}

\begin{lemma}\label{u4}
Let $R$ be a ring of finite dimension, $S = R[X]$ or $R[X, 1/X]$ and $\mathfrak{m}$ a maximal ideal of $S$ such that $\height(\mathfrak{m}) = \dim (S)$. Then $\mathfrak{m}$ contains a monic polynomial when $S = R[X]$ and a bimonic polynomial when $S = R[X, 1/X]$. 
\end{lemma}

\begin{proof}
By previous lemma \ref{u3}, $\mathfrak{n} = \mathfrak{m} \cap R$ is a maximal ideal of $R$. We have $\mathfrak{n}S \subsetneq \mathfrak{m}$.     Now $S/\mathfrak{n}S$ is a PID and $\overline{\mathfrak{m}}$ is generated by a monic polynomial when $S = R[X]$ and a bimonic polynomial when $S = R[X, 1/X]$. So we are through.
\end{proof}

The following is an easy consequence of Lemma \ref{u4}. The proof for $S = R[X]$ is contained in (\cite{brecosta}, Lemma 1).

\begin{lemma}\label{u5}
Let $R$ be a ring of finite dimension and $S = R[X]$ or $R[X, 1/X]$. Let $T$ denote the multiplicative closed set generated by monic polynomials when $S = R[X]$ and that generated by bimonic polynomials when $S = R[X, 1/X]$. Then $\dim( T^{-1}S) = \dim(S) - 1$.
\end{lemma}

The following is well known Horrocks theorem (\cite{lam}, Chapter V, Section \S 2,  Affine Horrocks 2.2, Supplement 2.3, Proposition 2.5).

\begin{theorem}\label{u10}
Let $S = R[X]$ or $R[X, 1/X]$ and $P$ a projective $S$-module. Let $f \in S$ be a monic when $S = R[X]$ and a bimonic when $S = R[X, 1/X]$. Then  
$P_f$ is extended from $R$ if and only if $P$ is extended from $R$. In particular $P_f$ is a free $S_f$-module if and only if $P$ is a free $S$-module.
\end{theorem} 

If $R$ is a zero dimensional ring then $\dim(R[X]) = \dim(R[X, 1/X]) = 1$. So Lemma \ref{u5}, Theorem \ref{u10}, Corollary \ref{P9e1} together and an induction argument on the number of variables give the following. See (\cite{cpnm}, Proposition 3.2) for a different argument.

\begin{lemma}\label{u11}
Let $R$ be a zero dimensional ring. Then any projective module over $ R[X_1, \ldots, X_n,  Y_1^{\pm 1}, \ldots,  Y_m^{\pm 1}]$ is free.
\end{lemma} 


\section{Action of transvection}
Let $S \rightarrow R$ be a ring homomorphism and $P$ a projective $R$-module. We shall say that $P$ is extended from $S$ if $P = Q \otimes_{S} R$ for some projective $S$-module $Q$. The following is an analogue of the elementary action on free modules for projective modules.

\begin{definition}{\rm(Transvections)}\label{ce1}
Let $P$ be a projective module over $R$, $p \in P, \pi \in P^* = \Hom(P, R)$ such that $\pi(p) = 0$. Let $\pi_p \in \End(P)$ be defined by $\pi_p(x) = \pi(x)p, x \in P$. Clearly $\pi_p^2 = 0$. An automorphism of the form $1 + \pi_p$ is called a transvection of $P$ if either $p \in \Um(P)$ or $\pi \in \Um(P^*)$. The group generated by transvections of $P$ is a normal subgroup of $\Aut(P)$ and is denoted by $\Trans(P)$. Let $I$ be an ideal of $R$. An automorphism of the form $1 + \pi_p$ is called a transvection relative to $I$ if either $p \in IP$ or $\pi \in IP^*$. The subgroup generated by relative transvections is denoted by $\Trans(P, I)$.

Let $P = R \oplus Q$. Let $\pi = pr_1$ i.e. projection on the first coordinate and  $p = (0, q) \in R \oplus Q$. Then $(1 + \pi_p) (a, x) = (a, x + aq)$. Again if $\pi = (0, \phi), \phi \in Q^*$ and $p = (1,0)$, then  $(1 + \pi_p )(a, x) = (a + \phi(x), x)$. Transvections of $R \oplus Q$ of these forms are called elementary transvections and the group generated by them is denoted by $\ETrans(P)$. When $q \in IQ$ or $\phi \in IQ^*$ then the transvections are called the relative elementary transvections with respect to an ideal $I$ and the group generated by them is denoted by $\ETrans(IP)$. The normal closure of $\ETrans(IP)$ in  $\ETrans(P)$ is denoted by $\ETrans(P, I)$.

We say that unimodular elements $v = (a, p), v' = (a', p') \in \Um(P, I)$ are in the same elementary orbit if we have $\tau \in \ETrans(P, I)$ such that $\tau (v) = v'$. We denote this by $v \sim_E v'$. 
\end{definition}

The following is an easy consequence of Corollary \ref{P8}.

\begin{theorem}\label{ce2}
Let $\Spec(R) = V(s) \sqcup D(s), s \in R$ such that each of $V(s)$ and $D(s)$ has dimension at most $d$ and $P = R \oplus Q $ a projective module of rank $r \geq d + 2$. Then  $\ETrans(P)$ acts transitively on $\Um(P)$.
\end{theorem}

\begin{definition}{\rm(Excision algebra and Excision module)}\label{ce3}
If $I$ is an ideal of $R$, one constructs the ring $R \oplus I$ with multiplication defined by $(a,  i)(b,  j) = (ab,  aj+bi+ij)$. We have two ring homomorphisms $\mathfrak{f}, pr_1 : R \oplus I \twoheadrightarrow R$ defined by $\mathfrak{f}(a, i) = a + i,  pr_1(a, i) = a$ respectively. Both $\mathfrak{f},  pr_1$ have a section $ \mathfrak{i}: R \hookrightarrow R \oplus I$ defined by $\mathfrak{i}(a) = (a, 0)$ satisfying $\mathfrak{f} \circ \mathfrak{i} = 1_R,  pr_1 \circ \mathfrak{i} = 1_R$. It is easy to see that if $I$ is  finitely generated then $R \oplus I$ is an $R$ algebra of dimension same as that of $R$ {\rm(see \cite{cpm}, Proposition 3.1)}. We call $R \oplus I$ the excision algebra of $R$.

Let $P$ be a finitely generated projective $R$-module. Then  we call $P \oplus IP = (R \oplus I) \otimes_{\mathfrak{i}} P$ which is a finitely generated projective $R \oplus I$-module, the Excision module. Note that $(P \oplus IP)^* = P^* \oplus IP^*$ as $P$ is a projective $R$-module.
\end{definition}

\begin{remark}\label{ce4}
 Let $Q = R \oplus P$ be a projective module over $R$. We have an obvious map  $\mathfrak{F} = R\otimes_{\mathfrak{f}}{ - }  : \Um(Q \oplus IQ, 0 \oplus I) \rightarrow \Um(Q, I)$ defined by $\mathfrak{F}((a_1, p_1), (a_2, p_2)) = (a_1 + a_2, p_1 + p_2)$ for $(a_1, p_1) \in Q$ and $(a_2, p_2) \in IQ$. Let  $q = (1 +  i, p) \in \Um(Q, I)$
and  $\phi = (1 +  j, \psi) \in \Um(Q^*, I)$
such that $\phi(q) = 1$.  We define $\tilde{q} = ( (1, 0), (i, p)) \in Q \oplus IQ$ and $\tilde{\phi} = ( (1, 0), (j, \psi)) \in (Q^ *\oplus IQ^*)$. Clearly $\tilde{\phi}(\tilde{q}) = (1, 0)$. So $\tilde{q} \in \Um (Q \oplus IQ, 0 \oplus I)$ and $\mathfrak{F}(\tilde{q}) = q$ i.e. $\mathfrak{F}$ is a surjective map.

We shall say $\tilde{q}, \tilde{\phi}$ to be a lift of $q, \phi$ respectively. If there exists a transvection $\tilde{\tau}\in \ETrans(Q \oplus IQ)$ such that $\tilde{\tau}(\tilde{q}) = (1, 0)$, then we can modify $\tilde{\tau}$ and assume that $\tilde{\tau} \in \ETrans(Q \oplus IQ, 0 \oplus I)$. This is possible because the map $pr_1 : R \oplus I \twoheadrightarrow R$ has a section $\mathfrak{i}$ and due to the following Lemma \ref{ce5}. Now applying $R \otimes_{\mathfrak{f}}{-}$ we have $\tau \in \ETrans(Q, I)$ such that $\tau(q) = (1, 0)$. 

Therefore, to show that the action of $\ETrans(Q, I)$ on $ \Um(Q, I)$ is transitive, it is enough to show that the action of $\ETrans(Q \oplus IQ)$ on $\Um(Q \oplus IQ)$ is transitive.
If $q \in \Um(Q, I)$, then $q \in \Um(Q, J)$ for some finitely generated sub-ideal $J$ of $I$. So to prove that the action of the group of transvections on $\Um(Q, I)$ is transitive, we may assume without loss of generality that $I$ is finitely generated and therefore $\dim(R \oplus I) = \dim(R)$ {\rm(\cite{cpm}, Proposition 3.1)}.
\end{remark}

\begin{lemma}\label{ce5}
Let $ P = S \oplus Q$ be a projective $S$-module. Let $ f : R \twoheadrightarrow S$ be a ring homomorphism which admits a section $g : S \rightarrow R$ such that $fg = 1$. Let $I = ker(f)$ and $\mathfrak{v} = (a, q) \in \Um(R \otimes_g P, I)$ such that $\tau(\mathfrak{v}) = (1, 0)$ for some transvection $\tau \in \ETrans(R \otimes_g P)$. Then there exists $\tau' \in \ETrans(R \otimes_g P, I)$ such that $\tau'(\mathfrak{v}) = (1, 0)$.
\end{lemma} 

\begin{proof}
We have $R \otimes _g P = R \oplus R \otimes_g Q$. For notational convenience we write $r \otimes s \in R \otimes_g S$ as $rs$ and $r \otimes_g q \in R \otimes Q$ as $rq$. For $(xa, yq) \in \Um(R \otimes_g P), (a, q) \in P, x, y \in R$ we define $e_{12}(r \phi)(xa, yq) = (xa + ry\phi(q), yq), r \in R, \phi \in Q^*$ and $e_{21}(rq')(xa, yq) = (xa, yq + rxaq'), r \in R, q, q' \in Q$. Since $Q$ is a projective $S$-module we have $(R \otimes_g Q)^* = R \otimes_g Q^*$. Therefore,  $\ETrans( R \otimes_g P)$ is generated by transvections of the above mentioned type. 

Now $e_{12}(r \phi) = e_{12}((r - gf(r)) \phi) e_{12}(gf(r)\phi) = \alpha \beta$ where $\alpha = e_{12}((r - gf(r))\phi) \in \ETrans ( IR \otimes_g P)$ and $\beta = e_{12}(gf(r)\phi)$ is extended from a transvection of $P$.  Similarly  $e_{21}(rq) = e_{21}((r - gf(r))q)e_{21}(gf(r)q)$ has a decomposition of such type.

Now $\tau(\mathfrak{v}) = (1, 0)$ gives $(S \otimes_f \tau) (1, 0) = (1, 0)$ as $S \otimes_f (\mathfrak{v}) = (1, 0)$. Therefore, replacing $\tau$ by $\{R\otimes_g S \otimes_f \tau\}^{-1}\tau$  we may assume that $S \otimes_f \tau = 1$. We write $\tau = \prod_{k= 1}^ne_{i_kj_k}(*) =\prod_{k= 1}^n \alpha_k \beta_k$ where $\alpha_k \in \ETrans ( IR \otimes P)$ and $\beta_k$'s are extended from transvections of $P$. Let $\gamma_i = \beta_1\beta_2 \ldots \beta_{i - 1}$. Then $\gamma_i$'s are extended and $\tau = \prod_{k= 1}^n \gamma_k \alpha_k \gamma_k^{-1} \times \gamma_{n + 1}$. We see that $S \otimes_f \gamma_{n +1} = 1$ which gives $\gamma_{n +1} = 1$ as $\gamma_{n + 1}$ is extended. Thus after this modification we have $\tau \in \ETrans(R \otimes_g P, I)$ such that $\tau(\mathfrak{v}) = (1, 0)$.
\end{proof}

\begin{lemma}\label{ce6}
Let $R$ be a ring such that $\Um_{n + 1}(R[X]) = e_1\E_{n + 1}(R[X]), n \geq 1$. Let $v = (f_0, f_1, \ldots, f_n) \in \Um_{n + 1}(R[X, 1/X])$ such that $lc(f_0)$ is a unit. Then $v$ is completable to a matrix in $\E_{n +1}(R[X, 1/X])$. 
\end{lemma}

\begin{proof}
Since $lc(f_0) = 1$, we can find $g_i \in R[X, 1/X]$ such that $f_i - g_i f_0 \in R[X]\ \forall\ i \geq 1$ and $lc(f_1 - g_1 f_0) = 1$. So  subtracting $g_i f_0$ from $f_i$, we assume that $f_i \in R[X]\ \forall \ i \geq 1$ and $lc(f_1) = 1$. Now we subtract a suitable multiple of $f_1$ from $f_0$ and assume that $f_0 \in R[X]$ also. Thus after suitable elementary operations we have $v \in \Um_{n + 1}(R[X, 1/X])$, $f_i \in R[X]\ \forall \ i \geq 0$ and $f_1(0) = 1$. Therefore, $v \in \Um_{n + 1}(R[X])$. Now $v$ is completable to a matrix in $\E_{n +1}(R[X])$ by the given hypothesis.
\end{proof}

The following result is proved in (\cite{abed}, Proposition 3.6) for the Laurent polynomial ring $R[X, 1/X]$ in the absolute case  .

\begin{lemma}\label{ce7}
Let $R$ be a zero dimensional ring and $I$ an ideal of $R$. Let $S = R[X]$ or $R[X, 1/X]$. Then $\Um_n(S, I) = e_1\E_n(S, I), n \geq 2$.
\end{lemma}

\begin{proof}
By Remark \ref{ce4} we only consider the absolute case. We can also assume that $R$ is a reduced ring. Let $v = (f_1, f_2 \ldots, f_n) \in \Um_n(S)$. We need to show that $v$ is elementarily completable.

\noindent
{\bf Case 1:} $S = R[X]$. 

We shall induct on the total number of coefficients  $N$ of $f_i$'s for arbitrary zero dimensional ring. If $N = 1$ then one of the coordinates of $v$ is zero. So we are done.  Let $deg (f_0)$ be the minimum among the degrees of the coordinates $f_i$.
 If $hc(f_0)$ is a unit. Then by the division algorithm we can easily reduce $N$ and induction prevails. So we assume that $a = hc(f_1)$ is a non-unit. By Lemma \ref{P3} we have $b \in ann (a)$ such that $xa + b = 1$. Now $R \cong R/(ax) \times R/(b)$, each factor being zero dimensional. In $R/(ax)$ we have  $\bar{a} = 0$ as it is killed by the unit $\bar{b}$. So the image $\bar{v}\in \Um_n(R/ax)$ is  elementarily completable by the induction hypothesis. On the other hand the image $\bar{a}$ is a unit in $R/(b)$. So by the division algorithm we can reduce the total number of coefficients of $\bar{v}\in \Um_n(R/(b))$ to less than $N$. Thus the image $\bar{v}\in \Um_n(R/(b))$ is also  elementarily completable. So $v$ is  elementarily completable.

\noindent
{\bf Case 2:} $S = R[X, 1/X]$.

As before we shall induct on the total number of coefficients $N$ of $f_i$'s for arbitrary zero dimensional ring. Let $a = lc(f_1)$. If $a$ is a unit then we are done by Lemma \ref{ce6} and the previous case. So we assume that $a$ is a non-unit.  Then as before $xa + b = 1$ for some $b \in ann(a)$ and $R \cong R/(ax) \times R/(b)$. Now $\bar{a}$ is a unit in $R/(b)$. So $\bar{v} \in \Um_n(R/(b))$ is  elementarily completable by Lemma \ref{ce6}. In $R/(ax)$ we have $\bar{a} = 0$. So $\bar{v} \in \Um_n{R/(ax)}$ is  elementarily completable by the induction hypothesis. Therefore, $v$ is  elementarily completable.
\end{proof}

\begin{proposition}\label{ce8}
Let $R$ be a ring of dimension $d$, $J$ the Jacobson  radical of $R$ and $P$ a projective module of rank $r \geq d$ over $R[X]$. Let $Q = R[X]^2 \oplus P$. Then  $\ETrans(Q, J)$ acts transitively on $\Um (Q, J)$.
\end{proposition}

\begin{proof}
By Lemma \ref{ce7}, it is enough to assume that $r \geq max \{1, d\}$. Let  $(f, g, p) \in \Um(Q, J)$. We shall show that $(f, g, p)$ can be sent to $(1, 0, 0)$ by the $\ETrans(Q, J)$ action by induction on the degree of $f$. If $deg (f) = 0$, then the assertion is obvious as $f$ is a unit and $1$ modulo $J$. So we assume that  $deg (f) \geq 1$ and the leading coefficient of $f$ is $a \in J$. Now going modulo $a$ we have by the induction hypothesis ${\tau} \in \ETrans(Q, J)$ such that  $\tau (f, g, p) \equiv (1, 0, 0) \vpmod a$. By an argument of Roitman (See \cite{roit}, Theorem 5) we can modify $\tau$ suitably such that $\tau (f, g, p) = (f', g', p') \equiv (1, 0, 0) \vpmod a$ and the leading coefficient of $f'$ is $a^l, l \geq 1$.

We recall the argument for the reader's convenience. Let $\tau = \tau_1\tau_2 \ldots \tau_n$, $\tau_i$'s are elementary transvections. If $\tau_i$ is of the form $(f, g, p) \rightarrow (f + hg, g, p)$ then we replace it by the composition $(f, g, p) \rightarrow (f, g + afX^n, p) \rightarrow (f + (h +  aX^n)( g + afX^n),g + afX^n, p)$, $n > \{deg(f), deg(g), deg(h)\}$. If $\tau_i$ is of the form  $(f, g, p) \rightarrow (f + \phi(p), g, p)$ then we shall replace it by $(f, g, p) \rightarrow (f, g + aX^n, p) \rightarrow (f + \phi(p), g + aX^n, p) \rightarrow (f + \phi(p) + aX( g + aX^n), g + aX^n, p)$, $n > \{deg(f), deg(g), deg(\phi(p))\}$. In other cases we do not  disturb $\tau_i$. Note that after such modifications still ${\tau} \in \ETrans(Q, J)$. We can also assume that the degree of $f'$ is sufficiently large i.e. $f' \not \in R$.

Now $\overline{R[X]} = R[X]/(f')$ is integral over $R_a$ and $a\overline{R[X]} = J\overline{R[X]} = \overline{R[X]}$ as $\bar{a}$ is a unit in $\overline{R[X]}$. We have $\dim (\overline{R[X]}) = \dim( R_a) \leq d - 1$. Thus by Theorem \ref{ce2} going modulo $f'$ we can send $(\bar{g'}, \bar{f'})$ to $(1, 0)$ by the action of elementary transvections. Lifting these transvections we  can change  $ (f', g', p')$ by suitable elementary operations (by the $\ETrans(aQ)$ action) such that    $ (f', g', p') \in  \Um(Q, aR[X])$ and satisfies $g' \equiv 1 \vpmod {f'}, p' \equiv 0 \vpmod {f'}$. Let $g' = 1 + f'h, h \in R[X], h \equiv -1 \vpmod a$ and $f' = 1 + af''$. Then $(f', g', p')$ can be sent to $(1, 0, 0)$ by the $\ETrans(Q, aR[X])$ action as follows $(f', g', p') \rightarrow (f', 1, p') \rightarrow (1, 1, p') \rightarrow (1, 1 + h, p') \rightarrow (1, 0, 0)$.
\end{proof}

Now we shall establish the following analogue of the above theorem for Laurent polynomial ring.

\begin{proposition}\label{ce9}
Let $R$ be a  ring of dimension $d$, $J$ the Jacobson  radical of $R$ and $P$ a projective module of rank $r \geq d$ over $R[X, 1/X]$. Let $Q = R[X, 1/ X]^2 \oplus P$ . Then  $\ETrans(Q, J)$ acts transitively on $\Um (Q, J)$.
\end{proposition}

\begin{proof}
By Lemma \ref{ce7} it is enough to assume that $r \geq max \{1, d\}$. Let $(f,g, p) \in \Um(Q, J)$. We consider the following cases.

\noindent
{\bf Case 1}: {\rm($(f, g, p) \equiv (1, 0, 0) \vpmod a,  lc(f) = a^m, hc(f) = a^n, m, n \geq 1$, $ a \in J$)}

We see that $R[X, 1/X]/(f)$  is integral over $R_a$ and $a\overline{R[X, 1/X]} = J\overline{R[X, 1/X]} = \overline{R[X, 1/X]}$. So the result follows using Theorem \ref{ce2} and arguments given in the last paragraph of the previous Proposition \ref{ce8}.
 
\noindent
{\bf Case 2}: {\rm($(f, g, p) \equiv (1, 0, 0)  \vpmod a, a \in J$ and  $hc(f) = a^n, n \geq 1$)}

We shall prove this case  by induction on the number of nonzero coefficients of $g$. If this number is zero then $g = 0$. So $(f, g, p)$ can be sent to $(1, 0, 0)$ by the $\ETrans(Q, J)$ (in fact by the $\ETrans(Q, (a))$) action.  

Now we assume that $g \not= 0 $. Let $lc(g) = ab$. Now going modulo $ab$ by the induction hypothesis we have $\tau \in \ETrans(Q, J)$ such that $\tau (f, g, p) \equiv  (1, 0, 0) \vpmod {ab}$. We shall modify $\tau$ such that $\tau(f, g, p) = (f', g', p') \equiv (1, 0, 0) \vpmod {ab}, hc(f') = (ab)^n, lc(f') = (ab)^m, m, n \geq 1$. The result will then follow from  Case 1. We describe the method below. It is again similar to Roitman's argument. 

 We shall first consider the action of transvections $(f, g, p) \rightarrow (f, g + ab^{n + 1}X^Nf, p)  \rightarrow (f + ab(X^{N} + X^{-N}) (g + ab^{n + 1}X^Nf), g + ab^{n + 1}X^N f, p) $. Here $N \gg 0$. Note that by these actions $(f, g, p)$ does not change modulo $ab$. So we assume that $hc(f) = (ab)^n, lc(f) = (ab)^m, n, m \geq 1, f \equiv 1 \vpmod a$. 

Let $\tau = \tau_1\tau_2 \ldots \tau_n$, $\tau_i$'s are elementary transvections. If $\tau_i$ is of the form $(f, g, p) \rightarrow (f + hg, g, p)$ we replace it by the composition $(f, g, h) \rightarrow (f, g + ab(X^N + X^{-N})f, p) \rightarrow (f  + \{h + ab(X^N + X^{-N})\}\{ g + ab(X^N + X^{-N})f\} , g + ab(X^N + X^{-N})f, p), N \gg 0$. If $\tau_i$ is of the form $(f, g, p) \rightarrow (f + \phi(p), g, p)$ we replace it by the composition $(f, g, h) \rightarrow (f, g + ab(X^N + X^{-N})f, p) \rightarrow (f  + \phi(p), g + ab(X^N + X^{-N})f, p)  \rightarrow  (f  + \phi(p) + ab\{g + ab(X^N + X^{-N})f\}, g + ab(X^N + X^{-N})f, p), N \gg 0$. We won't change $\tau_i$ in other cases. Note that after this modification we still have $\tau \in \ETrans(Q, J)$.

\noindent
{\bf Case 3}: {\rm(General Case)}

We shall induct on the number of nonzero coefficients of $f$. If the number is one then $f \equiv 1 \vpmod J$ is a unit. So  $(f, g, p)$ can be sent to $(1, 0, 0)$ by the $\ETrans(Q, J)$ action. 

Without loss of generality we assume that $f$ has at least one positive degree term. Let $hc(f) = a \in J$. Going modulo $a$ by the induction hypothesis we may assume that there exists $\tau \in  \ETrans(Q, J)$ such that  $\tau (f, g, p) \equiv  (1, 0, 0) \vpmod {a}$. Now we modify $\tau$ in the manner as described in the previous Proposition \ref{ce8} such that $\tau(f, g, p) = (f', g', p')$, $(f', g', p') \equiv (1, 0, 0)  \vpmod a$ and $hc(f') = a^n, n \geq 1$. Therefore, we are done by Case  \nolinebreak 2.
\end{proof}

\begin{corollary}\label{ce9e1}
Let $R$ be a semi local ring of dimension $d$ with Jacobson radical $J$, $S = R[X]$ or $R[X, 1/X]$ and $P$ a projective module of rank $r \geq d$ over $S$. Let $Q = S^2 \oplus P$. Then  $\ETrans(Q)$ acts transitively on $\Um (Q)$. 
\end{corollary}

\begin{proof}
Let $(f, g, p) \in \Um(Q)$. Note that $S/JS$ is a polynomial or a Laurent polynomial ring in $X$ over a product of finite number of fields. So $Q/JQ$ is a free $S/JS$-module of rank $r + 2$. Therefore, we have a transvection $\bar{\tau} \in \ETrans(Q/JQ)$ such that $\bar{\tau}(\bar{f}, \bar{g}, \bar{p}) = (1, 0, 0)$. We lift $\bar{\tau}$ to a transvection $\tau \in \ETrans(Q)$. We have $\tau(f, g, p) = (f', g', p')$, $(f', g', p') \in  \Um (Q, J)$. The result will now follow from Propositions  \ref{ce8}, \ref{ce9}. 
\end{proof}

The following is proved in (\cite{elso}, Theorem 4.8). It follows by argument in (\cite{lam}, Chapter III, Section \S 2, Proposition 2.3, Theorem 2.4, 2.5) with the aid of (\cite{bak}, Proposition 3.1).

\begin{lemma}{\rm(Local Global Principle)}\label {ce10}
Let $Q = R \oplus P$ be a projective $R$-module of rank $r \geq 3$ and $v(X) = \Um(Q[X])$. Suppose $v(X)_{\mathfrak{m}} \sim_E v(0)_{\mathfrak{m}}$ for all maximal ideals $\mathfrak{m}$ of $R$. Then $v(X) \sim_E v(0)$.
\end{lemma}

\begin{lemma}\label{cee101}
Let $P$ be a projective $R$-module of rank at least two and $(a, p) \in \Um(R \oplus P)$. Suppose $a$ is a unipotent element. Then $(a, p)$ can be sent to $(1, 0)$ by the action of transvections.
\end{lemma}
 
\begin{proof}
Let $a = 1 + n$ for $n \in \sqrt{0}$. Clearly $a$ is a unit. So $(a, p) \sim_E (a, 0)$. Therefore, it is enough to show that $(a, 0)$ can be sent to $(1, 0)$ by the action of elementary transvections. Now $v(X) = ( 1 + nX, 0) \in \Um(R[X] \oplus P[X])$ since $1 + nX$ is a unit in $R[X]$. For each maximal ideal $\mathfrak{m}$ of $R$,  $P_{\mathfrak{m}}$ is free. So $v(X)_{\mathfrak{m}}$ and $v(0)_{\mathfrak{m}}$ are in the same elementary orbit. Therefore, by Lemma \ref{ce10} $v(X)$ and $v(0)$  are in the same orbit. So $v(1) = (a, 0)$ can be sent to $(1, 0)$ by the action of elementary transvections. 
\end{proof}

\begin{proposition}\label{ce11}
Let $R$ be a ring of dimension $d$ and $I$ an ideal of $R$. Let $P$ be a projective $R[X]$-module extended from $R$ of rank $r \geq d + 1$. Let $Q = R[X] \oplus P$. Then $\ETrans(Q, I)$ acts transitively on $\Um (Q, I)$.
\end{proposition}

\begin{proof}
By Corollary \ref{P9}, P has a unimodular element. So $Q = R[X]^2 \oplus P'[X]$ for some projective $R$-module $P'$ of rank $r -1$. By Lemma \ref{ce7} we may assume that $r \geq \{2, d + 1\}$. By Remark \ref{ce4} it is enough to consider the absolute case i.e. $I = R$. 

It is also enough to consider that $R$ is local by the Local Global Principle Lemma \ref{ce10}. The result follows now from Corollary \ref{ce9e1}.
\end{proof}

\begin{proposition}\label{ce12}
Let $R$ be a ring of dimension $d$ and $I$ an ideal of $R$. Let $P$ be a projective  $R[X, 1/X]$-module extended from $R$ of rank $r \geq d + 1$. Let $Q = R[X, 1/X] \oplus P$. Then $\ETrans(Q, I)$ acts transitively on $\Um (Q, I)$.
\end{proposition}

\begin{proof}
By Corollary \ref{P9}, $Q = R[X, 1/X]^2 \oplus P'[X, 1/X]$ for some projective $R$-module $P'$ of rank $r -1$. By Lemma \ref{ce7}, we only need to consider the case  $r \geq \{2, d + 1\}$. 
By Remark \ref{ce4}, it suffices to consider only the absolute case i.e. $I = R$. 

Let $(f, g, p) \in \Um(Q)$. We shall induct on $N$ = the total number of coefficients of $f$ and $g$. If $N = 1$ then at least one of $f$ and $g$ is zero. So we are through. If both $hc(f), hc(g)$ are units then by the division algorithm we can easily reduce $N$ and induction prevails. So we assume that $hc(f) = a$, a non-unit. 

Going modulo $a$ by the induction hypothesis we may assume that there exists $\tau \in  \ETrans(Q)$ such that  $\tau (f, g, p) \equiv  (1, 0, 0) \vpmod {a}$. Now we modify $\tau$ in the manner as described in Proposition \ref{ce8} such that $\tau(f, g, p) = (f', g', p')$ and $(f', g', p') \equiv (1, 0, 0) \vpmod a$ and $hc(f') = a^n, n \geq 1$. So we may start with the assumption that $(f, g, p) \equiv (1, 0, 0)  \vpmod a$, $a$ is a non-unit and  $hc(f) = a^n, n \geq 1$. 
Again by similar argument as in Case 2, Proposition \ref{ce9} we reduce to the case when  $(f, g, p) \equiv (1, 0, 0) \vpmod a$, $lc(f) = a^m, hc(f) = a^n, m, n \geq 1$ and $a$ is a non-unit.

Now $\overline{R_{1 + aR}[X, 1/X]} = R_{1 + aR}[X, 1/X]/ (f)$ is integral over $R_{a(1 +aR)}$ and therefore has dimension at most $d - 1$. So by Theorem \ref{P6} we have $q \in P'[X, 1/X]$ such that $p + gq \in \Um( \overline{P'_{1 + aR}[X, 1/X]})$. So by a suitable action of transvection we may assume that $1 + ax \in (f) + O(p)$ for some $x \in R$. In particular there exists a monic polynomial $h \in \{(f) + O(p)\} \cap R[X]$ such that $h(0) = 1$. Now adding a suitable multiple of $h$ to $g$ we may assume that $g \in R[X]$ is monic and $g(0) = 1$. 

We add suitable multiples of $g$ to $f$ and $p$ to have $(f, g, p) \in R[X]^2 \oplus P'[X]$, $g(0) = 1$ such that  $(f, g, p) \in \Um(R[X, 1/X]^2 \oplus P'[X, 1/X])$. Therefore, $(f, g, p) \in  \Um(R[X]^2 \oplus P'[X])$ and the result follows from Proposition \ref{ce11}.
\end{proof}

If we translate the proof of our Propositions \ref{ce11} and \ref{ce12} to free case, we shall essentially  find a simplified proof of (\cite{yengui}, Theorem 5), (\cite{abed}, Theorem 3.12).

\begin{corollary}\label{ce13}
Let $R$ be a ring of dimension $d$ and $I$ an ideal of $R$. Then for $n \geq 2$ we have $\Um_n(R[X], I)= e_1\E_n(R[X], I)$ and $\Um_n(R[X, 1/X], I) = e_1\E_n(R[X, 1/X], I)$.
\end{corollary}

The following is an analogue of (\cite{lind}, Lemma 1.1) whose proof is essentially the same.

\begin{lemma}\label{ce14}
Let $P$ be a projective $R$-module of rank $r$. Assume that $s$ is a non-nilpotent such that $P_s$ is free. Then there exists $p _1, p_2, \ldots, p_n \in P, \phi_1, \phi_2, \ldots, \phi_r \in P^* = \Hom(P, R)$ and $t \in \NN$ such that 
\begin{itemize}
\item[1.] $(\phi_i(p_j)) = diagonal (s^t, s^t, \ldots, s^t)$.

\item[2.]  $s^tP \subset F$ and $s^tP^* \subset G$ with $F = \sum_{i = 1}^r Rp_i$ and $G = \sum _{i = 1}^r R\phi_i$.
\end{itemize}
\end{lemma}


\begin{definition}{\rm(Lindel Pair)}\label{ce16}
Let $P$ be a projective $R$-module and $s \in R$. We shall call a pair $(P, s)$ Lindel pair if either $P_s = 0$ i.e. $s$ is a nilpotent element or $P_s$ is a free $R_s$-module with $p _1, p_2, \ldots, p_n \in P, \phi_1, \phi_2, \ldots, \phi_r \in P^*$ satisfying the following conditions.

\begin{itemize}
\item[1.] $(\phi_i(p_j)) = diagonal (s, s, \ldots, s)$.

\item[2.]  $sP \subset F$ and $sP^* \subset G$ with $F = \sum_{i = 1}^r Rp_i$ and $G = \sum _{i = 1}^r R\phi_i$.
\end{itemize}
\end{definition}

If $(P, s)$ is a Lindel pair then $(P, st)$ is also a Lindel pair for any $t \in R$. Also if $\phi: R \rightarrow S$ is a ring homomorphism, then $(S \otimes_R P, \phi(s))$ is  a Lindel pair. If $R^n$ is a free module of rank $n$, then $(R^n \oplus P, s)$ is a Lindel pair. If $P_s$ is free then  by Lemma \ref{ce14} we have that $(P, s^t)$ is a Lindel pair for some integer $t$. The following is an analogue of (\cite{dhorajia}, Lemma 3.10).

\begin{lemma}\label{ce15}
Let $P$ be a projective $R$-module of rank $r \geq 2$ and $(P, s)$ a Lindel pair. Let  $\E_{r + 1}(R, sR)$ acts transitively on $\Um_{r + 1}(R, sR)$. Then for any $(a, p) \in \Um(R \oplus P, s^2)$, there exists $\tau \in \ETrans(R \oplus P)$ such that $\tau(a, p) = (1, 0)$.
\end{lemma}

\begin{proof}
If $s$ is a nilpotent then we are done by Lemma \ref{cee101}. So we assume that $s$ is a non-nilpotent element. Then $P_s$ is a free module with $p _1, p_2, \ldots, p_n \in P, \phi_1, \phi_2, \ldots, \phi_r \in P^*$ satisfying the conditions given in Definition \ref{ce16}.

Let $Q = R \oplus P$. Since $p \in s^2P$, we have $p = c_1p_1 + c_2p_2 + \ldots + c_rp_r$, $c_i \in sR$. Now $(a, c_1, c_2, \ldots, c_r) \in \Um_{r + 1}(R_s)$ as $P_s$ is a free module with basis $p_1, p_2, \ldots, p_r$. Also $a \equiv 1 \vpmod s$. So $(a, c_1, c_2, \ldots, c_r) \in \Um_{r + 1}(R, s)$. By hypothesis we have $\varepsilon \in \E_{r + 1}(R, sR)$ such that $\varepsilon (a, c_1, c_2, \ldots, c_r)^t = e_1$. Now  $\E_{r + 1}(R, sR)  \subset \E_{r + 1}^1(R, sR)$ (see \cite{vdk}, Lemma 2.2) which gives $\E_{r + 1}(R, sR)  \subset \E_{r + 1}^1(R, sR)^T$. Therefore, $\varepsilon$ is a product of elementary matrices of the form $E_{1j}(sx)$ and $E_{i1}(y)$, $x, y \in R$.

Note that any transvection on $Q$ can be written as $\tau_{\phi} = \begin{pmatrix} 1 & \phi \\ 0 & 1 \end{pmatrix}, \phi \in P^*$ and $\tau_q =  \begin{pmatrix} 1 & 0 \\ q & 1 \end{pmatrix}, q \in P$. Here $\tau_{\phi}(a, p) = (a + \phi(p), p)$ and $\tau_q(a, p) =(a, p + aq)$.  We see that $E_{1j}(sx)(a, c_1, c_2, \ldots, c_n)^T= (a + sxc_j, c_1, c_2, \ldots, c_n)^T$ and $\tau_{x\phi_j}(a, c_1p_1 + c_2p_2 + \ldots + c_rp_r) = (a + sxc_j, c_1p_1 + c_2p_2 + \ldots + c_rp_r)$. So $E_{1j}(sx)$ corresponds to $\tau_{x\phi_j}$. Similarly $E_{i1}(y)$ corresponds to $\tau_{yp_i}$.

Now in the expression of $\varepsilon$ we replace the elementary matrices by corresponding transvections without changing their order and call the resulting product of transvections as $\tau$. Clearly $\tau \in \ETrans(R \oplus P)$ and $\tau(a, p) = (1, 0)$. 
\end{proof}

 
%

\begin{lemma}\label{ce17}
Let $R$ be a ring of dimension $d$, $S = R[X]$ or $R[X, 1/X]$ and $P$  a projective $S$-module of rank $r \geq d + 1$. Suppose $s, t \in R$ such that $s + t$ is a unit in $R$ and $(P, s)$ is a Lindel pair. Let $v = (a, p) \in \Um(S \oplus P)$. Then $v \sim_E e_1$ if $v_t \sim _E e_1$. In particular $\ETrans(S \oplus P)$ acts transitively on $\Um(S \oplus P)$ if $\ETrans(S_t \oplus P_t)$ acts transitively on $\Um(S_t \oplus P_t)$. 
\end{lemma}

\begin{proof}
Due to Lemmas  \ref{u11} and \ref{ce7}, it is enough to assume that $r \geq \{2, d + 1\}$. In the quotient ring $R/s^2R$,  $\bar{t}$ is a unit. We have $\bar{\tau} \in \ETrans(S/s^2S \oplus P/s^2S)$ such that $\bar{\tau}(\bar{v}) = (\bar{1}, 0)$. Lifting  $\bar{\tau}$ to a transvection $\tau \in \ETrans(S \oplus P)$ we have $\tau(v) = v'$, $v' \in \Um(R \oplus P, s^2)$. Now by Corollary \ref{ce13} and Lemma \ref{ce15}, $v'$ can be sent to $(1, 0)$ by the action of transvections. 
\end{proof}

We are now ready to prove the Main Theorem of this section.

\begin{theorem}\label{ce18}
Let $R$ be a ring of dimension $d$, $I$ an ideal of $R$ and $S = R[X]$ or $R[X, 1/X]$. Let $P$ be a finitely generated projective $S$-module of rank $r \geq d$ and $Q = S^2 \oplus P$. Then $\ETrans(Q, I)$ acts transitively on $\Um(Q, I)$.
\end{theorem}

\begin{proof}
By Lemmas \ref{u11}, \ref{ce7}, we only need to consider the case  $r \geq \{1, d\}$. It suffices to consider only the absolute case i.e. $I = R$ because of Remark \ref{ce4}. Let $v = (f, g, p) \in \Um(Q)$.

For each minimal prime $\mathfrak{p}$ of $R$, $P_{\mathfrak{p}}$ is a free $S_{\mathfrak{p}}$-module. So we have $s_{\mathfrak{p}} \in R - \mathfrak{p}$ such that $(P, s_{\mathfrak{p}})$ is a Lindel pair. Let $J$ be the ideal generated by all such $s_{\mathfrak{p}}$. Then clearly height of $J$ is at least one as $J$ is not contained in any minimal prime ideal of $R$. We shall prove the result by induction on the dimension of the ring $R$.

If $\dim(R) = 0$, the result is obvious due to Lemmas  \ref{u11} and \ref{ce7}. We have $\dim(R/J) \leq \dim(R) - 1$. So by the induction hypothesis we have $\tau_1 \in \ETrans(Q)$ such that $\tau_1(v) = w =  (f', g', p') \in \Um(Q, J)$. Now $JR_{1 + J}$ is contained in the Jacobson radical of $R_{1 +J}$. Therefore, by Propositions \ref{ce8}, \ref{ce9} we obtain a transvection $\tau_2 \in \ETrans(Q_{1 + J})$ such that $\tau_2(w) = e_1$. This means that we have $s_0 \in 1 + J$ such that $w_{s_0} \sim_E e_1$.


Now $1 - s_0 \in J$. So we have $s_1, s_2, \ldots, s_n \in J$ such that $\sum_{i = 0}^n s_i = 1$ and $(P, s_i)$ is a Lindel pair for all $i \geq 1$. Let $t _i = s_0 + s_1 + \ldots + s_i$. In the ring $R_{t_i}$ we have $t_{i - 1} + s_{i} = t_{i}$ is a unit. We also have $(P_{t_i}, s_{i}/1), i \geq 1$ is a Lindel pair. So by Lemma \ref{ce17} we have $w_{t_i} \sim_E e_1$ if $w_{t_it_{i - 1}} \sim_E e_1$  in particular $w_{t_{i - 1}} \sim_E e_1$ for $i \geq 1$.

Now $s_0 = t_0$. So  $w_{t_0} \sim_E e_1$ which gives $w_{t_n} \sim_E e_1$. But $t_n = 1$. Therefore, we are done.
\end{proof}
\section{Existence of unimodular element}
In this section we shall investigate when a projective module over $R[X]$ or $R[X, 1/X]$ has a unimodular element. Our results are again similar to the corresponding results for noetherian rings. 


\begin{proposition}\label{u71}
Let $R$ be a ring of finite  dimension such that its Jacobson radical $\mathfrak{J}$ has height at least one. Let $S = R[X]$ or $R[X, 1/X]$. Suppose $P$ is a  projective $S$-module of rank $r \geq \dim (S)$. Then the following hold.

\begin{itemize}

\item[1.] If $S = R[X]$, then the natural map $\Um(P) \rightarrow \Um(P/XP)$ is surjective.

\item[2.] If $S = R[X, 1/X]$, then the natural map $\Um(P) \rightarrow \Um(P/(X - 1)P)$ is surjective.
\end{itemize}
In particular $P$ has a unimodular element. 
\end{proposition}

\begin{proof}
Our proofs for both the cases $S = R[X], R[X, 1/X]$ are similar. So we shall only consider the case when $S = R[X, 1/X]$.

We choose $\bar{p} \in \Um(P/(X - 1)P), p \in P$. Then $(X - 1, p) \in \Um(S \oplus P)$. Since $r \geq \dim(S/\mathfrak{J}S) + 1$, by Corollary \ref{P8} we can add a multiple of $(X -1)$ to $p$ to assume that $p$ is unimodular modulo $\mathfrak{J}$ i.e. $ O(p) + \mathfrak{J}S = S$. Therefore, we have $f \in \mathfrak{J}S$ such that $(f, p) \in \Um(S \oplus P)$. This gives  $(f(X - 1), p) \in \Um(S \oplus P)$. Let $T$ denote the multiplicative closed set of all bimonic polynomials in $S$. Then $\dim (T^{-1}S) = \dim (S) - 1$ by Lemma \ref{u5}. By Theorem \ref{P6} we have $p_1 \in P$ such that $q = p + f(X - 1)p_1 \in \Um(T^{-1}P)$. 

This means that the ideal $O(q)$ contains a bimonic polynomial. We  have $O(q) +  \mathfrak{J}S = S$ as $p \equiv q \vpmod {\mathfrak{J}}$. This gives  $O(q) \cap R +  \mathfrak{J} = R$ by Lemma \ref{u1}.  So $q \in \Um(P)$. Note that $q \equiv p \vpmod {(X - 1)}$. So we are done.
\end{proof}

%
%
%
%
%

The next lemma follows from the well known Quillen Splitting Lemma whose proof is essentially contained in (\cite{qu}, Lemma  1, Theorem  1).

\begin{lemma}\label{u8}
Let $R$ be a ring and $P$ a projective module over $R$. Let $s, t \in R$ be such that $Rs + Rt = R$. Let $\sigma(T)$ be a $R_{st}[T]$ automorphism of $P_{st}[T]$ such that $\sigma(0) = id$. Then $\sigma(T) = \alpha(T)_t \beta(T)_s$, where $\alpha(T)$ is a $R_s[T]$ automorphism of $P_s[T]$ such that $\alpha(T) \equiv id \vpmod {tT}$ and $\beta(T)$ is a $R_t[T]$ automorphism of $P_t[T]$ such that $\beta(T) \equiv id \vpmod {sT}$.  Moreover if $J$ is an ideal of $R$ such that $\sigma \equiv id \vpmod J$, then we may also assume that $\alpha \equiv id \vpmod J$, $\beta \equiv id \vpmod J$.
\end{lemma}

The following is an easy consequence of the above.

\begin{lemma}\label{u81}
Let $R$ be a ring and $P$ a projective module over $R[T]$. Let $s, t \in R$ be such that $Rs + Rt = R$. Assume that $P_{st}$ is extended from $R$.  Let $\sigma(T)$ be a $R_{st}[T]$ automorphism of $P_{st}$ such that $\sigma(0) = id$. Then $\sigma(T) = \alpha_t \beta_s$, where $\alpha$ is a $R_s[T]$ automorphism of $P_s$ such that $\alpha \equiv id \vpmod {tT}$ and $\beta$ is a $R_t[T]$ automorphism of $P_t$ such that $\beta \equiv id \vpmod {sT}$.
\end{lemma}

There is no splitting lemma for automorphisms of projective modules over Laurent polynomial rings. However we have the following weaker version.

\begin{lemma}\label{u82}
Let $R$ be a ring and $P$ a projective module over $R[T, 1/T]$. Let $s, t \in R$ be such that $Rs + Rt = R$. Assume that $P_{st} = R_{st}[T, 1/T] \oplus Q[T, 1/ T]$ for some projective $R_{st}$-module $Q$. Let $\sigma(T) \in \ETrans(P_{st})$ such that $\sigma(1) = id$. Then $\sigma(T) = \alpha_t \beta_s$, where $\alpha$ is a $R_s[T, 1/T]$ automorphism of $P_s$ such that $\alpha \equiv id \vpmod {t(T - 1)}$ and $\beta$ is a $R_t[T, 1/T]$ automorphism of $P_t$ such that $\beta \equiv id \vpmod {s(T - 1)}$.
\end{lemma}

\begin{proof}
Let $e_{12}(\phi) , e_{21}(q) \in \ETrans(P_{st})$ be defined as $e_{12}(\phi) (a, p) = (a + \phi(p), p)$, $\phi \in Q^{*}[T, 1/T]$ and $e_{21}(q)(a, p) = (a, p + aq), q \in Q[T, 1/T]$. We have the decomposition $\sigma(T) = \prod_k \gamma_k^{-1} e_{i_kj_k}((T - 1)g_{i_kj_k}) \gamma_k$ where $\gamma_k \in \ETrans(P_{st}/ (T - 1)P_{st})$ and $g_{i_kj_k} \in Q[T, 1/T]$ or $Q^{*}[T, 1/T]$ according as $(i_k, j_k) = (2, 1)$ or $(1, 2)$. Now $\hat{\sigma}(X) = \prod_k \gamma_k^{-1} e_{i_kj_k}(X(T -1)g_{i_kj_k}) \gamma_k$ is a transvection acting on $P_{st}[X]$ such that  $\hat{\sigma}(1) = \sigma$. By Quillen Splitting Lemma \ref{u8} we have  $\hat{\sigma}(X) =  \hat{\alpha}(X)_t \hat{\beta}(X)_s$ where $\hat{\alpha}(X)$ is an automorphism of $P_s[X]$ such that $\hat{\alpha}(X) \equiv id \vpmod {tX(T - 1)}$ and $\hat{\beta}(X)$ is an automorphism of $P_t[X]$ such that $\hat{\beta}(X) \equiv id \vpmod {sX(T - 1)}$. Let $\hat{\alpha}(1) = \alpha$ and $\hat{\beta}(1) = \beta$. Then $\sigma = \alpha_t \beta_s$, $\alpha$ is an automorphism of $P_s$ such that $\alpha \equiv id \vpmod {t(T - 1)}$ and $\beta$ is an automorphism of $P_t$ such that $\beta \equiv id \vpmod {s( T - 1)}$. We are done.
\end{proof}

\begin{definition}\label{u14}\rm{(Quillen Ideal)}
Let $R$ be a commutative ring and $P$ a projective $R[T]$-module. Let $J(R, P) \subset R$ consist of all those $a \in R$ such that $P_a$ is extended from $R_a$. It follows from $($\cite{qu}, Theorem 1$)$ that $J(R, P)$ is an ideal and $J(R, P) = \sqrt{J(R, P)}$. From Lemma \ref{u11} it is easy to see that $\height J(R, P) \geq 1$. We call $J(R, P)$ the Quillen ideal of $R$.
\end{definition}

The following is a generalization of Proposition \ref{u71}.

\begin{theorem}\label{u15}
Let $R$ be a ring of finite dimension. Let $P$ be a finitely generated projective $R[X]$- module of rank $r \geq \dim (R[X])$. Then  the natural map $\Um(P) \rightarrow \Um(P/XP)$ is surjective. In particular $P$ has a unimodular element.
\end{theorem}

\begin{proof}
Let $\bar{p} \in \Um(P/(X - 1)P)$ for  $p \in P$. We want to find $q \in \Um(P)$ whose image in $P/ XP$ is $\bar{p}$. If $\dim(R) = 0$, then $P$ is free by Lemma \ref{u11} and the theorem  follows obviously. So we assume that $\dim (R) \geq 1$ and $r \geq 2$. 

 Let $I = J(R, P)$ denote the Quillen ideal of $R$. We have $IR_{1 + I} \subset \mathfrak{J}$, where $\mathfrak{J}$ is the Jacobson radical of $R_{1+I}$. So $\height(\mathfrak{J}) \geq 1$. 
Therefore, by Proposition \ref{u71} we have $q_1 \in \Um(P_{ 1 + I})$ such that $q_1 \equiv p_{1 + I} \vpmod {X}$. We choose $s \in I$ such that $q_1 \in \Um(P_{ 1 + sR})$ and $q_1 \equiv p_{1 + sR} \vpmod {X}$.

If $s$ is nilpotent, then each element of $1 + sR$ is a unit. So  $q_1 \in \Um(P)$, $q_1 \equiv p \vpmod {X}$ and we are done. Therefore, we assume that $s$ is not a nilpotent element in $R$. We have the following patching diagram.
$$\begin{CD}
P @>>> P_s \\ @VVV  @VVV\\ P_{(1 + sR)} @>>> P_{s(1 + sR)}
\end{CD}$$

 $P_s$ is a projective module extended from $R$. So we have $q_2 \in \Um(P_s)$ such that $q_2 \equiv p_{s} \vpmod {X}$. Now $P_{s(1 + sR)}$ is  an extended projective module of rank $r \geq \dim(R[X]) \geq \dim(R) + 1 \geq \dim(R_{s(1 + sR)}) + 2$. So by Corollary \ref{P9} we have  $P_{s(1 + sR)} = R_{s(1 + sR)}[X]^2 \oplus P'$ for some extended projective module $P'$ of rank $r - 2$. By Proposition \ref{ce11} we have $\tau \in \ETrans(P_{s(1 + sR)}, (X))$ such that $\tau( {q_1}_s) = {q_2}_{(1 + sR)}$. We choose $t \in R$ such that $(s, t) = R$, $P_{st}$ is extended from $R_{st}$, $q_1 \in \Um(P_t)$, $q_1 \equiv p_{t} \vpmod {X}$, $\tau \in \ETrans(P_{st}, (X))$ and $\tau( {q_1}_s) = {q_2}_{t}$.

By Lemma \ref{u81} we have a splitting $\tau(X) = \alpha_t \beta_s$, where $\alpha$ is a $R_s[X]$ automorphism of $P_s$ such that $\alpha(X) \equiv id \vpmod {tX}$ and $\beta$ is a $R_t[X]$ automorphism of $P_t$ such that $\beta(X) \equiv id \vpmod {sX}$. Therefore, $\tau( {q_1}_s) = {q_2}_{t}$ gives ${\beta(q_1)}_s = { \alpha^{-1}(q_2)}_t$. Patching ${\beta(q_1)}$ and $\alpha^{-1}(q_2)$ we have $q \in \Um(P)$ such that $q_s = \alpha^{-1}(q_2)$ and $q_t = {\beta(q_1)}$. Note that  $q _s \equiv q_2 \equiv p_s \vpmod{X}$ and $q _t \equiv q_1 \equiv p_t \vpmod{X}$. Therefore, $q \equiv p \vpmod{(X}$ and we are through.
\end{proof}

The method of the proof of Theorem \ref{u15} won't work for projective modules over Laurent polynomial rings as the notion of Quillen ideal is not available. We therefore take a different approach. If $P$ is a projective $R$-module and $I$ an ideal of $R$, then $p \in \Um(P)$ gives $\bar{p} \in \Um(P/IP)$. We shall call $p$ a lift of $\bar{p}$.

\begin{proposition}\label{u9}
Let $R$ be a ring and $P$ a projective $R[T, 1/T]$-module of rank $r \geq \dim(R[X, 1/X])$. Let $s, t \in R$ be such that $Rs + Rt = R$. Suppose $P_s$ is extended from $R_s$. Let $p \in P$ be such that $\bar{p} \in \Um(P/(X - 1)P)$. Assume that $\bar{p}_t \in \Um(P_t/(X - 1)P)$ has a lift in $\Um(P_t)$. Then $\bar{p}$ has a lift in  $\Um(P)$. In particular $\Um(P) \rightarrow \Um(P/(X - 1)P)$ is surjective if  $\Um(P_t) \rightarrow \Um(P_t/(X - 1)P)$ is surjective.  
\end{proposition}

\begin{proof}
We have $\bar{p} \in \Um(P/(X - 1)P)$ for  $p \in P$. We want to find $q \in \Um(P)$ whose image in $P/ (X - 1)P$ is $\bar{p}$. We have the following patching diagram

$$\begin{CD}
P @>>> P_s \\ @VVV  @VVV\\ P_{(1 + sR)} @>>> P_{s(1 + sR)}
\end{CD}$$

Since $P_s$ is extended from $R$ we have $q_1 \in \Um(P_s)$ such that ${q_1} \equiv p_s \vpmod {(X - 1)}$. Now $1 + sR$ is a multiplicative closed set containing a multiple of $t$. So we have  $q_2 \in \Um(P_{(1 + sR)})$ such that ${q_2} \equiv p_{1 + sR} \vpmod {(X - 1)}$ by our hypothesis.

Now $P_{s(1 + sR)}$ is extended from $R_{s(1 + sR)}$ of rank $r \geq \dim (R[X]) \geq \dim(R) + 1 \geq \dim(R_{s(1 + sR)}) + 2$. So by Corollary \ref{P9} we have  $P_{s(1 + sR)} = R_{s(1 + sR)}[X, 1/X]^2 \oplus P'$ for some extended projective module $P'$ of rank $r - 2$.

We have a transvection $\tau \in \ETrans(P_{s(1 + sR)},(X - 1))$ such that $\tau({q_2}_s) = {q_1}_{(1 + sR)}$ by Proposition \ref{ce12}. We can find $t \in R$ such that $(s, t) = R$, $P_{st}$ is extended from $R_{st}$, $q_2  \in \Um(P_t)$, ${q_2} \equiv p_t \vpmod {(X - 1)}$,  $\tau \in \ETrans(P_{st},(X - 1))$,  $\tau({q_2}_s) = {q_1}_{t}$. 

By Lemma \ref{u82} we have $\tau = {\tau_1}_t{\tau_2}_s$ where $\tau_1 \in Aut(P_s, t(X - 1))$ and $\tau_2  \in Aut(P_t, s(X - 1))$.
Therefore, $\tau({q_2}_s) = {q_1}_{t}$ gives ${\tau_2(q_2)}_s = {\tau_1^{-1}(q_1)}_t$. Patching $\tau_2(q_2) \in \Um(P_t)$ and $\tau_1^{-1}(q_1)  \in \Um(P_s)$ we have $q \in \Um(P)$ such that $q_s = \tau_1^{-1}(q_1)$ and $q_t = \tau_2(q_2)$. Clearly $q _s \equiv q_1 \equiv p_s \vpmod{(X -1)}$ and $q _t \equiv q_2 \equiv p_t \vpmod{(X -1)}$. Hence $q \equiv p \vpmod{(X -1)}$ and we are done.
\end{proof}

We shall now prove an analogue of Theorem \ref{u15} for Laurent polynomial rings.
 
\begin{theorem}\label{u13}
Let $R$ be a ring of finite dimension. Let $P$ be a finitely generated projective $R[X, 1/X]$-module of rank $r \geq \dim (R[X, 1/X])$. Then the natural map $\Um(P) \rightarrow \Um(P/(X - 1)P)$ is surjective. In particular $P$ has a unimodular element. 
\end{theorem}

\begin{proof}
If $\mathfrak{p}$ is any minimal prime ideal  of $R$, then $P_{\mathfrak{p}}$ is a free $R_{\mathfrak{p}}[X, 1/X]$-module by Lemma \ref{u11}. So we have $s_{\mathfrak{p}} \in R - \mathfrak{p}$ such that $P_{s_{\mathfrak{p}}}$ is a free $R_{s_{\mathfrak{p}}}[X, 1/X]$-module. Let $I$ be the ideal generated by all $s_{\mathfrak{p}}$, $\mathfrak{p}$ a minimal prime. Then height of $I$ is at least one as it is not contained in any minimal prime ideal of $R$. Let $\bar{p} \in \Um(P/(X - 1)P)$ for  $p \in P$.

By Proposition \ref{u71} the map $\Um(P_{1 + I}) \rightarrow \Um(P_{1 + I}/(X - 1)P_{1 + I})$ is surjective. So we have $s_0 \in 1 + I$ such that $\bar{p}_{s_0} \in \Um(P_{s_0}/(X - 1)P_{s_0})$ has a lift $q \in \Um(P_{s_0})$ under the natural map $\Um(P_{s_0}) \rightarrow \Um(P_{s_0}/(X - 1)P_{s_0})$. Now $1 - s_0 \in I$ gives $s_1, \ldots, s_n \in I$ such that  $P_{s_i}, i \geq 1$ is free and $s_0 + s_1 + \ldots + s_n = 1$. Let $t_i = s_0 + s_1 + \ldots + s_i$. In the ring $R_{t_i}$, $s_{i} + t_{i - 1} = t_{i}$ is a unit. Also $P_{s_i t_{i}}, i \geq 1$ is free. So by Proposition \ref{u9},  $\bar{p}_{t_i} \in \Um(P_{t_i}/(X - 1)P_{t_i})$ has a lift in $\Um(P_{t_i})$ if $\bar{p}_{t_{i - 1}t_i } \in \Um(P_{t_{i - 1}t_i }/(X - 1)P_{t_{i - 1}t_i })$ has a lift in $\Um(P_{t_{i - 1}t_i })$ for $i \geq 1$. In particular $\bar{p}_{t_i} \in \Um(P_{t_i}/(X - 1)P_{t_i})$ has a lift in $\Um(P_{t_i})$ whenever $\bar{p}_{t_{i - 1}} \in \Um(P_{t_{i - 1} }/(X - 1)P_{t_{i - 1}})$ has a lift in $\Um(P_{t_{i - 1}})$  for $i \geq 1$.


Now $s_0 = t_0$. So  $q \in \Um(P_{t_0})$ is a lift of $\bar{p}_{t_0} \in \Um(P_{t_0}/(X - 1)P_{t_0})$. Therefore, $\bar{p}_{t_n} \in \Um(P_{t_n}/(X - 1)P_{t_n})$ has a lift in $\Um(P_{t_n})$. But $t_n = 1$. So we are done.
\end{proof}

Our Theorems \ref{u15}, \ref{u13} lead us to ask the following question.

\begin{question}\label{u16}
Let $R$ be a ring of finite dimension and $S = R[X_1, \ldots, X_n,  Y_1^{\pm 1}, \ldots,  \\ Y_m^{\pm 1}]$.  Let $P$ be a finitely generated projective $S$-module of rank $r > \dim(S) - (m + n)$. Then is the map $\Um(P) \rightarrow \Um(P/(X_1, \ldots, X_n, Y_1 - 1, \ldots, Y_m - 1)P)$ surjective? 
\end{question}

By Lemma \ref{u11} if $R$ is a zero dimensional ring, then any projective module over $ S = R[X_1, \ldots, X_n, Y_1^{\pm1}, \ldots, Y_m^{\pm1}]$ is free. Also $\dim(S) = m + n$ when $R$ is zero dimensional. So the answer to the above question is affirmative when dimension of the ring is zero.


\section{main results}
In this section we shall discuss our main results on cancellative nature of projective modules. All results are generalization of corresponding results for noetherian rings. We recall that a projective $R$-module $P$ is called cancellative if $Q \oplus P \cong Q \oplus P'$ for some projective $R$-modules $P', Q$ implies that $P \cong P'$. Equivalently  $P$ is called cancellative if $R^n \oplus P \cong R^n \oplus P', n \geq 1$ gives  $P \cong P'$. The following is easy.

\begin{lemma}\label{mr1}
Let $P$ be a projective $R$-module such that $Aut(R \oplus P)$ acts transitively on $\Um(R \oplus P)$. Then $R \oplus P \cong R \oplus P'$ gives $P \cong P'$.
\end{lemma}


\begin{theorem}\label{mr2}
Let $R$ be a ring of dimension $d$ and $S = R[X]$ or $R[X, 1/X]$. Let $P$ be a finitely generated projective $S$-module of rank $r$.
Then $P$ is cancellative in the following cases. 

\begin{itemize}
\item [(1)] $P$ has a unimodular element and $r \geq d + 1$.

\item [(2)] Rank of $P$ is at least equal to $\dim(S)$.

\end{itemize}
\end{theorem}

\begin{proof} 
In view of Lemma \ref{mr1}, (1) will follow from Theorem \ref{ce18}, (2) will follow from (1) and Theorems \ref{u15} and \ref{u13}.
\end{proof}

We don't know if our estimate is best possible. In particular we like to know the following.

\begin{question}\label{mr2e1}
Let $R$ be a ring of dimension $d$ and $S = R[X]$ or $R[X, 1/X]$. Let $P$ be a finitely generated projective $S$-modules of rank $r \geq d + 1$. 
Then does $P$ have a unimodular element. Is $P$ cancellative?
\end{question}

The following question  is towards a possible generalization of Theorem \ref{mr2} in several variables. 

\begin{question}\label{mr2e2}
Let $R$ be a ring of finite dimension  and $S = R[X_1, X_2, \ldots, X_n,  Y_1^{\pm 1},  Y_2^{\pm 1}\\,  \ldots,  Y_m^{\pm 1}]$.  Let $P$ be a finitely generated projective $S$-module of rank $r > \dim(S) - (m + n)$. Then is $P$ cancellative?
\end{question}

\begin{definition}{\rm(Strong S-ring)}\label{mr3}
A ring $R$ is said to be a strong S-ring if for any two consecutive prime ideals $\mathfrak{p} \subset \mathfrak{q}$ in $R$ the prime ideals $\mathfrak{p}[X] \subset \mathfrak{q}[X]$ are consecutive in $R[X]$.
\end{definition}

It has been shown in (\cite{malik}, Corollary 3.6) that if $R$ is of finite type over a Pr\"{u}fer domain then $R$ is a strong S-ring. By (\cite{kap}, Chapter 1, Section \S 5, Theorem 39) we have $\dim (R[X]) = \dim(R) + 1, dim R[X, 1/X] = \dim(R) + 1$. So the following is a trivial consequence of Theorems \ref{u15}, \ref{u13}, \ref{mr2} .

\begin{theorem}\label{mr3.5}
Let $R$ be a ring  of dimension $d$  of finite type over a Pr\"{u}fer domain and  $S = R[X]$ or $R[X, 1/X]$. Let $P$ be a finitely generated projective $S$-module of rank $r \geq d + 1$. Then the following hold.

\begin{itemize}
\item[1.] If $S = R[X]$, then the natural map $\Um(P) \rightarrow \Um(P/XP)$ is surjective.

\item[2.] If $S = R[X, 1/X]$, then the natural map $\Um(P) \rightarrow \Um(P/(X - 1)P)$ is surjective.
\end{itemize}
In particular $P$ has a unimodular element.  Moreover if $P'$ is another projective $S$-module of rank $r$ and $Q \oplus P \cong Q \oplus P'$ for some projective $S$-module $Q$, then $P \cong P'$.
\end{theorem}

Let $R$ be a  domain which is not a field and $T$  the multiplicative closed set generated by all the bimonic polynomials in $R[X, 1/X]$. Then by similar argument as in (\cite{brecosta}, Theorem 1) we can claim that $R$ is a Pr\"{u}fer domain of dimension one if and only if $T^{-1}R[X, 1/X]$ is also a Pr\"{u}fer domain of dimension one. Now by an induction argument using Theorem \ref{u10} we have the following.

\begin{theorem}\label{mr6}
Let $R$ be a Pr\"{u}fer domain of dimension one and  $S = R[X_1,  \ldots, X_n, Y_1^{\pm 1}\\, \ldots,  Y_m^{\pm 1}]$. Then any projective module over $S$ is extended from $R$.
\end{theorem}

We conclude this article with the following question.


\begin{question}\label{mr7}
Is Theorem \ref{mr6} true for Pr\"{u}fer domains of arbitrary dimension? When $m = 0$ the question is answered affirmatively in \cite{lesi}.
\end{question}

\begin{acknowledgement}
This paper is a part of my PhD thesis. I wish to thank my advisor Ravi A. Rao for continuous encouragement during the preparation of this article.  I am also thankful to SPM Fellowship, CSIR,  (SPM - 07 - /858(0051)/ 2008 - EMR - 1) for the financial support.
\end{acknowledgement}


\end{document}